\documentclass[12pt]{amsart}
\usepackage{amsmath,amsthm,amssymb,bbm,color,verbatim}
\input xy 
\xyoption{all}

\numberwithin{equation}{section}

\newcommand{\GCH}{{\rm GCH}}
\newcommand{\ZFC}{{\rm ZFC}}
\newcommand{\ZF}{{\rm ZF}}
\newcommand{\AC}{{\rm AC}}
\newcommand{\ORD}{\mathop{{\rm ORD}}}
\newcommand{\AD}{{\rm AD}}

\renewcommand{\P}{{\mathbb P}}
\newcommand{\Q}{{\mathbb Q}}

\newcommand{\Add}{\mathop{\rm Add}}
\newcommand{\Coll}{\mathop{\rm Coll}}
\newcommand{\forces}{\Vdash}
\newcommand{\decides}{\parallel}
\newcommand{\strongsubset}{\hspace{-0.15em}\begin{tabular}{c}$\subset$\\[-0.8 em]$\sim$\\\end{tabular}\hspace{-0.13em}}

\newcommand{\restrict}{\upharpoonright}
\newcommand{\concat}{\mathbin{{}^\smallfrown}}
\newcommand{\la}{\langle}
\newcommand{\ra}{\rangle}

\newcommand{\supp}{\mathop{\rm supp}}
\newcommand{\dom}{\mathop{\rm dom}}

\newcommand{\ot}{\mathop{\rm ot}\nolimits}
\newcommand{\cf}{\mathop{\rm cf}}

\newtheorem{theorem}{Theorem}
\newtheorem{lemma}[theorem]{Lemma}
\newtheorem{question}{Question}

\thanks{The first author's research was
partially supported by PSC-CUNY grants.}
\date{December 6, 2011
      (revised April 5, 2012)}
\begin{document}

\title[Consec. singular cardinals and the continuum function]{Consecutive singular cardinals and the continuum function}

\author[Arthur W. Apter]{Arthur W. Apter}
\address[Arthur W. Apter]{ 
Department of Mathematics, 
Baruch College of CUNY, 
One Bernard Baruch Way, 
New York, New York 10010 USA, and
Department of Mathematics, 
The Graduate Center of CUNY, 
365 Fifth Avenue,
New York, NY 10016 USA}
\email[A. W. ~Apter]{awapter@alum.mit.edu}
\urladdr{http://faculty.baruch.cuny.edu/aapter/}

\author[Brent Cody]{Brent Cody}
\address[Brent Cody]{ 
Department of Mathematics, 
The Graduate Center of CUNY, 
365 Fifth Avenue,
New York, NY 10016 USA} 
\email[B. ~Cody]{bcody@gc.cuny.edu} 
\urladdr{https://wfs.gc.cuny.edu/BCody/www/}

\maketitle

\begin{abstract}
We show that from a supercompact cardinal $\kappa$, there is a forcing extension $V[G]$ that has a symmetric inner model $N$ in which $\ZF+\lnot\AC$ holds, $\kappa$ and $\kappa^+$ are both singular, and the continuum function at $\kappa$ can be precisely controlled, in the sense that the final model contains a sequence of distinct subsets of $\kappa$ of length equal to any predetermined ordinal. We also show that the above situation can be collapsed to obtain a model of $\ZF+\lnot\AC_\omega$ in which either (1) $\aleph_1$ and $\aleph_2$ are both singular and the continuum function at $\aleph_1$ can be precisely controlled, or (2) $\aleph_\omega$ and $\aleph_{\omega+1}$ are both singular and the continuum function at $\aleph_\omega$ can be precisely controlled. Additionally, we discuss a result in which we separate the lengths of sequences of distinct subsets of consecutive singular cardinals $\kappa$ and $\kappa^+$ in a model of $\ZF$. Some open questions concerning the continuum function in models of $\ZF$ with consecutive singular cardinals are posed.

\end{abstract}

\section{Introduction}

In this paper we will be motivated by the question: Are there models of $\ZF$ with consecutive singular cardinals $\kappa$ and $\kappa^+$ such that ``$\GCH$ fails at $\kappa$'' in the sense that there is a sequence of distinct subsets of $\kappa$ of length greater than $\kappa^+$? Let us start by considering some known models of $\ZF$ that have consecutive singular cardinals.

Gitik showed in \cite{Gitik:AllUncountableCardinalsCanBeSingular} that from a proper class of strongly compact cardinals, $\langle\kappa_\alpha\mid\alpha\in\ORD\rangle$, there is a model of $\ZF+\lnot\AC_\omega$ in which all uncountable cardinals are singular. Essentially he uses a certain type of generalized Prikry forcing that simultaneously singularizes and collapses each $\kappa_\alpha$, thereby resulting in a model in which the class of uncountable well-ordered cardinals consists of the previously strongly compact $\kappa_\alpha$'s and their limits. In this model, every uncountable cardinal is singular, and for each $\alpha\in\ORD$ and for each limit ordinal $\lambda$, all cardinals in the open intervals $(\kappa_\alpha,\kappa_{\alpha+1})$ and $(\sup_{\beta<\lambda}\kappa_\beta,\kappa_\lambda)$ have been collapsed to have size $\kappa_\alpha$ and $\sup_{\beta<\lambda}\kappa_\beta$ respectively. Since each $\kappa_\alpha$ is a strong limit cardinal in the ground model, it follows that in Gitik's final model there is no cardinal $\kappa$ that has a sequence of distinct subsets of length greater than---or even equal to---$\kappa^+$.
(Of course, trivially, in any model of $\ZF$, for any cardinal $\kappa$,
there is always a $\kappa$-sequence of distinct subsets of $\kappa$ given
by the sequence of intervals $\la [\alpha,\kappa)\mid\alpha<\kappa\ra$.
This trivially also implies that,
for any $\beta \in (\kappa, \kappa^+)$,
there is a $\beta$-sequence of distinct subsets
of $\kappa$ as well.)
For similar reasons, the models constructed in \cite{Gitik:RegularCardinalsInModelsOfZF} and \cite{ApterDimitiouKoepke:TheFirstMeasurableCardinalCanBe} will also not have consecutive singular cardinals $\kappa$ and $\kappa^+$ with a sequence of distinct subsets of $\kappa$ of length even $\kappa^+$.

There has been a great deal of work, involving forcing over models of $\AD$, in which models are constructed having consecutive singular cardinals, as exemplified by \cite{Apter:ADAndPatterns}. However, in any model of $\AD$, no cardinal $\kappa<\Theta$ has a sequence of distinct subsets of length $\kappa^+$ let alone of longer length (see \cite{Steel:OutlineOfInnerModelTheory}). Thus, forcing over a model of $\AD$ does not seem to yield, in any obvious way, a model containing consecutive singular cardinals, $\kappa$ and $\kappa^+$, in which there is a sequence of distinct subsets of $\kappa$ of length $\kappa^+$.

In this article, we will show that from a supercompact cardinal, there are models of $\ZF+\lnot\AC$ that have consecutive singular cardinals, say $\kappa$ and $\kappa^+$, such that there is a sequence of distinct subsets of $\kappa$ of length equal to any predetermined ordinal. Indeed, we will prove the following.

\begin{theorem}\label{maintheorem}
Suppose $\kappa$ is supercompact, $\GCH$ holds, and $\theta$ is an ordinal. Then there is a forcing extension $V[G]$ that has a symmetric inner model $N\subseteq V[G]$ of $\ZF+\lnot\AC$ in which the following hold.
\begin{enumerate}
\item $\kappa$ and $\kappa^+$ are both singular with $\cf(\kappa)^N=\omega$ and $\cf(\kappa^+)^N<\kappa$.
\item $\kappa$ is a strong limit cardinal that is a limit of inaccessible cardinals.
\item There is a sequence of distinct subsets of $\kappa$ of length $\theta$.
\end{enumerate}
\end{theorem}

Let us remark here that property (3) in Theorem \ref{maintheorem} makes this result interesting, since none of the previously known models with consecutive singular cardinals discussed above satisfies it when $\theta\geq\kappa^+$. Since the definitions of ``strong limit cardinal'' and ``inaccessible cardinal'' generally do not make sense in models of $\lnot\AC$, let us explain why the assertion in Theorem \ref{maintheorem}
that (2) holds in $N$ makes sense.
It will be the case that $N$ and $V$ will have the same bounded subsets
of $\kappa$, and from this it follows that the usual definitions of
``$\kappa$ is a strong limit cardinal'' and ``$\delta<\kappa$ is an inaccessible cardinal'' make sense in $N$.

%For $\delta<\kappa$, the usual definition of $``\delta$ is inaccessible'' makes sense in $N$ because $N$ and $V$ will have the same bounded subsets of $\kappa$.

%need to explain what we mean by saying, (2) holds in $N$. For $\delta<\kappa$, the usual definition of $``\delta$ is inaccessible'' makes sense in $N$ because $N$ and $V$ will have the same bounded subsets of $\kappa$. Hence, saying that $\kappa$ is a limit of inaccessible cardinals in $N$ makes sense. Similaraly, saying that ``$\kappa$ is a strong limit cardinal'' holds in $N$ makes sense.

%For $\delta<\kappa$, we say that $\delta$ is an inaccessible cardinal in $N$ if and only if $\delta$ is an inaccessible cardinal in $V$. This makes sense because $N$ and $V$ will have the same bounded subsets of $\kappa$. Thus the assertion in Theorem \ref{maintheorem}(2) above, that $\kappa$ is a limit of inaccessible cardinals, makes sense in this context. Similarly, the usual definition for ``$\kappa$ is a strong limit cardinal'' makes sense in $N$ because $N$ and $V$ have the same bounded subsets of $\kappa$.

%To describe what (2) means in the model $N\models\lnot\AC$, it will suffice to explain what we mean by ``a strong limit cardinal'' below $\kappa$, in the model $N$. For $\delta<\kappa$, we say that $\delta$ is a strong limit cardinal in $N$ if and only if $\delta$ is a strong limit cardinal in $V$. This makes sense because $N$ and $V$ will have the same bounded subsets of $\kappa$.

Using the methods of \cite{Bull:SuccessiveLargeCardinals}, \cite{Apter:SomeResultsOnConsecutiveLargeCardinals}, and \cite{ApterHenle:RelativeConsistencyResultsViaStrongCompactness} we also obtain the following two results.

\begin{theorem}\label{theoremcollapse1}
Suppose $\kappa$ is supercompact, $\GCH$ holds, and $\theta$ is an ordinal. Then there is a model of $\ZF+\lnot \AC_\omega$ in which $\cf(\aleph_1)=\cf(\aleph_2)=\omega$, and there is a sequence of distinct subsets of $\aleph_1$ of length $\theta$.
%If it is consistent with $\ZFC$ that there is a supercompact cardinal, then for any ordinal $\theta$ there is a model of $\ZF+\lnot\AC_\omega$ in which $\aleph_1$ and $\aleph_2$ are both singular and there is a sequence of distinct subsets of $\aleph_1$ of length $\theta$.
\end{theorem}

\begin{theorem}\label{theoremcollapseomega}
Suppose $\kappa$ is supercompact, $\GCH$ holds, and $\theta$ is an ordinal. Then there is a model of $\ZF+\lnot\AC_\omega$ in which $\aleph_\omega$ and $\aleph_{\omega+1}$ are both singular with $\omega\leq\cf(\aleph_{\omega+1})<\aleph_\omega$, and there is a sequence of distinct subsets of $\aleph_\omega$ of length $\theta$.
%If it is consistent with $\ZFC$ that there is a supercompact cardinal, then for any ordinal $\theta$ there is a model of $\ZF+\lnot\AC_\omega$ in which $\aleph_\omega$ and $\aleph_{\omega+1}$ are both singular and there is a sequence of distinct subsets of $\aleph_\omega$ of length $\theta$. \marginpar{\tiny In Theorem \ref{theoremcollapseomega} we could say}
\end{theorem}

We note that in Theorem \ref{theoremcollapse1}, $\aleph_1$ and $\aleph_2$ can be replaced with $\delta$ and $\delta^+$ respectively, where $\delta$ is the successor of any ground model regular cardinal less than $\kappa$. Also, in Theorem \ref{theoremcollapseomega}, we note that $\aleph_\omega$ and $\aleph_{\omega+1}$ can be replaced by $\eta$ and $\eta^+$ respectively, where $\eta<\kappa$ can be any reasonably defined singular limit cardinal of cofinality $\omega$. We will return to these issues later.

Let us now give a brief outline of the rest of the paper. In Section \ref{sectionpreliminaries}, we include a definition of the basic forcing notion we will use and outline its important properties. In Section \ref{sectionmaintheorem}, we give a detailed proof of Theorem \ref{maintheorem}. In Section \ref{sectioncollapse}, we sketch the proofs of Theorem \ref{theoremcollapse1} and Theorem \ref{theoremcollapseomega}. In Section \ref{sectionadditional}, we discuss a result in which we separate the lengths of distinct subsets of consecutive singular cardinals, and we also pose some open questions.

\begin{comment}

Discuss Gitik's model in which every cardinal is singular as well as some other background results on consecutive singular/large cardinals and $\lnot AC$ models.

The following might be useful references for the introduction:

\begin{enumerate}
\item Levy proved in \cite{Levy:Independence} that it is consistent with $\ZF$ that $\aleph_1$ is singular. (Maybe I should only mention more recent results on consecutive singular cardinals?)
\item Gitik proved in \cite{Gitik:AllUncountableCardinalsCanBeSingular} that the consistency of the theory ``$\ZFC$ + there is a proper class of measurable cardinals'' implies the consistency of ``$\ZF$ + every uncountable cardinal is sinuglar.'' 
\item \cite{Bull:SuccessiveLargeCardinals}? Should we cite Bull's thesis?
\item \cite{Apter:SomeResultsOnConsecutiveLargeCardinals}
\end{enumerate}

\begin{theorem}\label{maintheoremconsistency}
The consistency of the theory ``$\ZFC$ + $\exists\kappa$ such that $\kappa$ is supercompact'' implies the consistency of ``$\ZF$ + $\lnot\AC_\omega$ + $\exists\kappa$ such that $\kappa$ and $\kappa^+$ are singular with cofinality $\omega$ and there is a $\lambda$-sequence of subsets of $\kappa$'' where $\lambda$ is an arbitrary cardinal.
\end{theorem}

\end{comment}

\section{Preliminaries}\label{sectionpreliminaries}

In this section, we will briefly discuss the various forcing notions used. If $\kappa$ is a regular cardinal and $\lambda$ is an ordinal, $\Add(\kappa,\lambda)$ denotes the standard partial order for adding $\lambda$ Cohen subsets to $\kappa$. If $\lambda>\kappa$ is an inaccessible cardinal, $\Coll(\kappa,{<}\lambda)$ is the standard partial order for collapsing $\lambda$ to $\kappa^+$ and all cardinals in the interval $[\kappa,\lambda)$ to $\kappa$.
For further details, we refer the reader to \cite{Jech:Book}. For a given partial order $\P$ and a condition $p\in\P$, we define $\P/p:=\{q\in\P\mid q\leq p\}$. If $\varphi$ is a statement in the forcing language associated with $\P$ and $p\in\P$, we write $p \parallel \varphi$ if and only if $p$ decides $\varphi$.

We will now review the definition and important features of supercompact
Prikry forcing and refer the reader to
\cite{Gitik:Handbook} or \cite{Apter:SuccessorsOfSingularCardinalsAnd}
for details. Suppose $\kappa$ is $\lambda$-supercompact and
that $U$ is a normal fine measure on $P_\kappa\lambda$
satisfying the Menas partition property
(see \cite{Menas:Paper}
for a definition and a proof of
the fact that if $\kappa$ is $2^\lambda$-supercompact,
then $P_\kappa\lambda$ has a normal fine measure
with this property).
For $P,Q\in P_\kappa\lambda$ we say that $P$ is \emph{strongly included} in $Q$ and write $P\strongsubset Q$ if $P\subseteq Q$ and $\ot(P)<\ot(Q\cap\kappa)$. We define \emph{supercompact Prikry forcing} $\P$ to be the set of all ordered tuples of the form $\la P_1,\ldots, P_n, A\ra$ such that
\begin{enumerate}
\item $P_1,\ldots,P_n$ is a finite $\strongsubset$-increasing sequence of elements of $P_\kappa\lambda$,
\item $A\in U$, and
\item for every $Q\in A$, $P_n\strongsubset Q$.
\end{enumerate}
Given $\la P_1,\ldots,P_n,A\ra,\la Q_1,\ldots,Q_m,B\ra\in\P$ we say that $\la P_1,\ldots,P_n,A\ra$ \emph{extends} $\la Q_1,\ldots,Q_m,B\ra$ and write $\la P_1,\ldots,P_n,A\ra\leq \la Q_1,\ldots,Q_m,B\ra$ if and only if 
\begin{enumerate}
\item $n\geq m$,
\item for each $k\leq m$, $P_k=Q_k$,
\item $A\subseteq B$, and
\item $\{P_{m+1},\ldots,P_n\}\subseteq B$.
\end{enumerate}
%Furthermore, we say that $\la P_1,\ldots,P_n,A\ra$ is a \emph{direct extension} of $\la Q_1,\ldots,Q_m,B\ra$ and write $\la P_1,\ldots,P_n,A\ra\leq^* \la Q_1,\ldots,Q_m,B\ra$ if and only if $\la P_1,\ldots,P_n\ra=\la Q_1,\ldots,Q_m\ra$ and $A\subseteq B$.

Since any two conditions of the form $\la P_1\ldots,P_n,A\ra$ and $\la P_1,\ldots, P_n, B\ra$ in $\P$ are compatible, one may easily show that $\P$ is $(\lambda^{<\kappa})^+$-c.c. Since $U$ satisfies the Menas partition property, it follows that forcing with $\P$ does not add new bounded subsets to $\kappa$. In the forcing extension by $\P$, $\kappa$ has cofinality $\omega$, and if $\lambda>\kappa$ then certain cardinals will be collapsed according to the following.
\begin{lemma}\label{lemmaprikryforcing}
Every $\gamma\in [\kappa,\lambda]$ of cofinality at least $\kappa$ (in V) changes its cofinality to $\omega$ in $V[G]$. Moreover, in $V[G]$, every cardinal in $(\kappa,\lambda]$ is collapsed to have size $\kappa$. 
\end{lemma}
%\noindent For a proof of the facts mentioned above as well as Lemma \ref{lemmaprikryforcing} one may consult Gitik's article in \cite{HandbookOfSetTheoryVolume2}.

\section{The Proof of Theorem \ref{maintheorem}}\label{sectionmaintheorem}

Now we will begin the proof of Theorem \ref{maintheorem}. We note that our proof amalgamates the methods used in \cite{ApterHenle:RelativeConsistencyResultsViaStrongCompactness} with those of \cite{Apter:SuccessorsOfSingularCardinalsAnd}.

\begin{proof}[Proof of Theorem \ref{maintheorem}]

Suppose $\kappa$ is supercompact and $\theta$ is an ordinal in some initial model $V_0$ of $\ZFC+\GCH$. We will show that there is a forcing extension of $V_0$ that has a symmetric inner model $N$ in which $\kappa$ and $\kappa^+$ are both singular with $\cf(\kappa)^N=\omega$ and $\cf(\kappa^+)^N<\kappa$, and there is a $\theta$-sequence of subsets of $\kappa$. By first forcing the supercompactness of $\kappa$ to be Laver indestuctible, as in \cite{Laver:MakingSupercompactnessIndestructible}, and then forcing with $\Add(\kappa,\theta)$, we may assume without loss of generality that $\kappa$ is supercompact and $2^\kappa=\theta$ in a forcing extension $V$ of $V_0$. Let $\lambda$ be a cardinal such that $\kappa<\lambda$ and $\cf(\lambda)^V<\kappa$. In $V$, let $\P$ be the supercompact Prikry forcing relative to some normal fine measure $U$ on $P_\kappa\lambda$ satisfying the Menas partition property. Let $G$ be $V$-generic for $\P$ and let $\langle P_n\mid n<\omega\rangle$ be the supercompact Prikry sequence associated with $G$; that is, $\langle P_n\mid n<\omega\rangle$ is the sequence of elements of $P_\kappa\lambda$ such that for each $n<\omega$, there is an $A\in U$ with $(P_1,\ldots,P_n,A)\in G$.

By Lemma \ref{lemmaprikryforcing}, it follows that in $V[G]$, the cofinality of $\kappa$ is $\omega$, and every ordinal in the interval $(\kappa,\lambda]$ has size $\kappa$. Furthermore, since the supercompact Prikry forcing adds no new bounded subsets to $\kappa$, it follows that $\kappa$ remains a cardinal in $V[G]$. We will now define a symmetric inner model $N\subseteq V[G]$ in which $\kappa^+=\lambda$, and we will argue that the conclusions of Theorem \ref{maintheorem} hold in $N$.

%\marginpar{\tiny Let us make a note concerning the chain condition of $\P$. Since any two conditions with the same stem are compatible it follows that any antichain of $\P$ has size at most $|P_\kappa\lambda|$. We assumed that $\cf^V(\lambda)<\kappa$ and from this it follows that $|P_\kappa\lambda|=\lambda^{<\kappa}=\lambda^+$. Thus $\P$ is $\lambda^{++}$-c.c. and hence may collapse $\lambda^+$. Regardless of this, $\lambda^+$ will be a cardinal in the symmetric inner model $N$ we will define below.}

In order to define $N$, we need to discuss a way of restricting
the forcing conditions in $\P$. First note that,
as in \cite{Apter:SuccessorsOfSingularCardinalsAnd},
for $\delta\in[\kappa,\lambda]$ a regular cardinal,
$U\restrict\delta:=U\cap P(P_\kappa\delta)$ is a normal fine measure on
$P_\kappa\delta$ satisfying the Menas partition property.
%\marginpar{\tiny In \cite[Proof of Theorem 1]{Apter:SuccessorsOfSingularCardinalsAnd}
%you say $\hat{U}$ can be chosen on $P_\kappa(2^\lambda)$
%so that $U=\hat{U}\restrict\lambda$ has the Menas partition proerty.
%So, does our $U\restrict\delta$ really have the Menas partition property?}
Let $\P_{U\restrict\delta}$ denote the supercompact Prikry forcing associated with $U\restrict\delta$. If $p=\la Q_1,\ldots Q_n,A\ra\in \P$ we define $p\restrict\delta:=\la Q_1\cap\delta,\ldots,Q_n\cap\delta,A\cap P_\kappa\delta\ra$ and note that $p\in \P_{U\restrict\delta}$. If $A\in P_\kappa\lambda$ we define $A\restrict\delta:=A\cap P_\kappa\delta$. The Mathias genericity criterion \cite{Mathias:OnSequencesGenericInTheSenseOfPrikry} for supercompact Prikry forcing yields that $r_\delta:=\langle P_n\cap\delta\mid n<\omega\rangle$ generates a $V$-generic filter for $\P_{U\restrict\delta}$. Indeed, $G\restrict\delta:=G\cap \P_{U\restrict\delta}$ is the generic filter for $\P_{U\restrict\delta}$ generated by $r_\delta$. $N$ is now defined informally as the smallest model of $\ZF$ extending $V$ which contains $r_\delta$ for each regular cardinal $\delta\in[\kappa,\lambda)$ but not the full supercompact Prikry sequence $r:=\langle P_n\mid n<\omega\rangle$.

We may define $N$ more formally as follows. Let $\mathcal{L}$ be the forcing language associated with $\P$ and let $\mathcal{L}_1\subseteq\mathcal{L}$ be the ramified sublanguage containing symbols $\check{v}$ for each $v\in V$, a unary predicate $\check{V}$ (interpreted as $\check{V}(\check{v})$ if and only if $v\in V$), and symbols $\dot{r}_\delta$ for each regular cardinal $\delta\in[\kappa,\lambda)$. We define $N$ inductively inside $V[G]$ as follows:
\begin{align*}
N_0&=\emptyset, \\ 
N_\delta&=\bigcup_{\alpha<\delta}N_\alpha\textrm{ for $\delta$ a limit ordinal}, \\ 
N_{\alpha+1}&=\{x\subseteq N_\alpha\mid \textrm{$x$ can be defined over $\langle N_\alpha,\in,c\rangle_{c\in N_\alpha}$} \\
& \hspace{1in}\textrm{using a forcing term $\tau\in\mathcal{L}_1$ of rank $\leq\alpha$}\}, \textrm{ and}\\ 
N&=\bigcup_{\alpha\in\ORD} N_\alpha.
\end{align*}

Standard arguments show that $N\models \ZF$.
As usual, each $\check{v}$ for $v\in V$ may be chosen so
as to be invariant under %any automorphism of $\P_U$.
any isomorphism $\Psi : \P/p \to \P/q$ for
$p, q \in \P$.
Further, terms $\tau$ mentioning only $\dot{r}_\delta$ may be chosen so
as to be invariant under %any automorphism of $\P_U$ which preserves
any isomorphism $\Psi : \P/p \to \P/q$ which preserves
the meaning of $r_\delta$.

The following lemma provides the key to showing that $N$ has the desired features.

\begin{lemma}\label{keylemma}
If $x\in N$ is a set of ordinals, then for some regular cardinal $\delta\in [\kappa,\lambda)$, $x\in V[r_\delta]$.
\end{lemma}

\begin{proof}

Let us note that the following proof of Lemma \ref{keylemma} blends ideas found in the proofs of \cite[Lemma 1.5]{Apter:SuccessorsOfSingularCardinalsAnd} and \cite[Lemma 2.1]{ApterHenle:RelativeConsistencyResultsViaStrongCompactness}. Let $\tau$ be a term in $\mathcal{L}_1$ for $x$. Suppose $\beta$ is an ordinal, $p\forces_{\P,V}\tau\subseteq\beta$, and $p\in G$. Since $\tau\in\mathcal{L}_1$, it follows that $\tau$ mentions finitely many terms of the form $\dot{r}_\delta$. Without loss of generality, we may assume that $\tau$ mentions a single $\dot{r}_\delta$. We will show that $x\in V[r_\delta]$. Let
$$y:=\{\alpha<\beta\mid\exists q\leq p\ (q\restrict\delta\in G\restrict\delta \textrm{ and } q\forces_{\P,V}\alpha\in\tau)\}.$$
We will show that $x=y$. Since it is clear that $y\in V[r_\delta]$, this will suffice. Suppose $\alpha\in x$, and choose $p'\leq p$ with $p'\in G$ such that $p'\forces_{\P,V}\alpha\in \tau$. Since $p'\restrict \delta\in G\restrict \delta$, we conclude that $\alpha\in y$. Thus, $x\subseteq y$. Now suppose $\alpha\in y$, and let $q\leq p$ with $q\restrict\delta\in G\restrict\delta$ and $q\forces_{\P,V}\alpha\in\tau$. There is a $q'\in G$ such that $q'\decides\alpha\in\tau$. If $q'\forces\alpha\in\tau$, then $\alpha\in x$ and we are done; thus we assume that $q'\forces\alpha\notin\tau$. Write $q=\langle Q_1,\ldots,Q_l, A\rangle$ and $q'=\langle Q_1',\ldots, Q'_m, A'\rangle$, where without loss of generality we assume that $l<m$. Since $q'\restrict\delta,q\restrict\delta\in G\restrict\delta$ and $l<m$, we know that $Q_i\cap\delta=Q_i'\cap\delta$ for $1\leq i \leq l$. Furthermore, there is some $q^*:=\langle Q_0\cap\delta,\ldots, Q_l\cap\delta,R^*_{l+1},\ldots,R^*_m,A^*\rangle\in G\restrict\delta$ extending $q\restrict\delta$ with $R^*_i=Q'_i\cap\delta$ for $l+1\leq i \leq m$ (to find such a condition one could just take a common extension of $q'\restrict\delta$ and $q\restrict\delta$ in $G\restrict\delta$ and then obtain the appropriate stem by throwing unwanted points back into the measure one set).
Now let us argue that there is a $q''\leq q$ in $\P$
such that $q''=\langle Q_0,\ldots,Q_l,S_{l+1},\ldots,S_m,A''\rangle$, and for $l+1\leq i\leq m$ we have $S_i\cap\delta=R^*_i= Q'_i\cap\delta$. Since $q^*\leq_{\P\restrict\delta} q\restrict\delta$, it follows by the definition of $\leq_{\P\restrict\delta}$ that for $l+1\leq i\leq m$, $R_i^*\in A\restrict\delta=A\cap P_\kappa\delta$, which implies $R^*_i= S_i\cap\delta$ for some $S_i\in A$. Also by the definition of $\leq_{\P\restrict\delta}$, we have $A^*\subseteq A\restrict\delta$, and since $q^*\in \P\restrict\delta$, we have $A^*=B\restrict\delta$ for some $B\in U$. Now let $A'':=A'\cap A\cap B$ and notice that $A''\restrict\delta\subseteq A^*$. Indeed we have $q''\restrict\delta\leq_{\P\restrict\delta} q^*$ and $q''\leq_\P q$. We let $q'''$ be the condition extending $q'$ defined by $q''':=\la Q_1',\ldots, Q_m', A''\ra$.

Now we define an isomorphism from $\P/q''$ to $\P/q'''$ that sends $q''$ to $q'''$ and fixes $\tau$. Let $\Psi:P_\kappa\lambda\to P_\kappa\lambda$ be the permutation defined by $\Psi(Q_i)=Q_i'$ and $\Psi(Q_i')=Q_i$ for $1\leq i\leq l$, $\Psi(S_i)=Q_i'$ and $\Psi(Q_i')=S_i$ for $l+1\leq i\leq m$, and $\Psi$ is the identity otherwise. This permutation induces a map $\Psi:\P/p''\to\P/p'''$ defined by $\Psi(\langle P_1,\ldots, P_n,C\rangle)=\langle\Psi(P_1),\ldots,\Psi(P_n),\Psi"C\rangle$. Note that since $\Psi$ fixes all but finitely many elements of $P_\kappa\lambda$, it follows that $\Psi"C\in U$. One may check that $\Psi$ is an isomorphism, and it easily follows that $\Psi(q'')=\langle Q_1',\ldots,Q_m',\Psi"A''\rangle =\langle Q_1',\ldots,Q_m',A''\rangle =q'''$.
Furthermore, since $\tau$ mentions only $\dot{r}_\delta$,
%which is a $\P\restrict\delta$-term,
since $$\langle Q_1\cap\delta,\ldots, Q_l\cap\delta,S_{l+1}\cap\delta, \ldots, S_m\cap\delta\rangle=\langle Q_1'\cap\delta,\ldots,Q_m'\cap\delta\rangle,$$
and since any condition $\langle Q_1,\ldots, Q_l, S_{l+1},\ldots,S_m, S_{m+1},\ldots, S_k, D\rangle$ extending $q''$ must have $S_{i}\notin\{Q_1,\ldots,Q_l,S_{l+1},\ldots, S_m, Q'_1,\ldots, Q'_m\}$ for $m+1\leq i\leq k$, it follows that $\Psi$ does not affect the meaning of $\tau$.
By extending $\Psi$ to the relevant
$\P$-terms, since $q''\forces \alpha\in\tau$, we have $\Psi(q'')\forces\Psi(\alpha)\in\Psi(\tau)$. This implies $\Psi(q'')=q'''\forces\alpha\in\tau$. This contradicts the fact that $q'''\leq q'\forces\alpha\notin\tau$.
\end{proof}

Since $V\subseteq N\subseteq V[G]$ and $\P$ does not add bounded subsets to $\kappa$, it follows that $N$ and $V$ have the same bounded subsets of $\kappa$. Thus, in $N$, $\kappa$ is a limit of inaccessible cardinals, and hence is also a strong limit cardinal.

We will now use Lemma \ref{keylemma} to show that $\lambda$, which was collapsed to have size $\kappa$ in $V[G]$, is a cardinal in $N$, and furthermore, $(\kappa^+)^N=\lambda$ and $\cf(\lambda)^N=\cf(\lambda)^V$.

Let us argue that if $\gamma\geq\lambda$ is a cardinal in $V$, then $\gamma$ remains a cardinal in $N$. Suppose for a contradiction that $\gamma$ is not a cardinal in $N$. Then there is a bijection from some $\alpha<\gamma$ to $\gamma$ which is coded by a set of ordinals in $N$. By Lemma \ref{keylemma}, there is a regular cardinal $\delta\in(\kappa,\lambda)$ such that the code
and hence the bijection are in $V[G\restrict\delta]$. This implies that $\gamma$ is not a cardinal in $V[G\restrict\delta]$. We will obtain a contradiction by using the chain condition of $\P_{U\restrict\delta}$ to show that $\gamma$ is a cardinal in $V[G\restrict\delta]$. Indeed, we will show that even though $\GCH$ may fail at $\kappa$ in $V$, the supercompact Prikry forcing $\P_{U\restrict\delta}$ is $\delta^+$-c.c. in $V$. As mentioned in Section \ref{sectionpreliminaries}, $\P_{U\restrict\delta}$ is $(\delta^{<\kappa})^+$-c.c. in $V$. Since $\GCH$ holds in $V_0$ we have $(\delta^{<\kappa})^{V_0}=\delta$, and since $\Add(\kappa,\theta)$ preserves cardinals and adds no sequences of ordinals of length less than $\kappa$, we conclude that $(\delta^{<\kappa})^V=(\delta^{<\kappa})^{V_0}=\delta$. This shows that $\P_{U\restrict\delta}$ is $\delta^+$-c.c. in $V$, and thus $\gamma$ is a cardinal in $V[G\restrict\delta]$, a contradiction.

For each regular cardinal $\delta\in(\kappa,\lambda)$, we have $V[G\restrict\delta]\subseteq N$, and this implies that $\cf^N(\kappa)=\omega$ and that every ordinal in $(\kappa,\lambda)$ which is a cardinal in $V$ is collapsed to have size $\kappa$ in $N$. Thus, we have $(\kappa^+)^N=\lambda$. Furthermore, since $N$ and $V$ agree on bounded subsets of $\kappa$, we see that $\cf^N(\lambda)=\cf^V(\lambda)<\kappa$. This shows that $\cf^N((\kappa^+)^N)=\cf^V(\lambda)<\kappa$,  and this implies that $N$ satisfies $\lnot \AC$. Since $V\subseteq N$, and since $(2^\kappa=\theta)^V$, it follows that there is a $\theta$-sequence of distinct subsets of $\kappa$ in $N$.

This completes the proof of Theorem \ref{maintheorem}.
\end{proof}

Let us emphasize: The fact that $\GCH$ can potentially
fail at $\kappa$ in $V$, depending on the size of $\theta$,
together with the cardinal preservation to $N$,
are the features of our construction that
set the results of this paper
apart from those previously discussed in the literature.

\section{The Proofs of Theorems \ref{theoremcollapse1} and \ref{theoremcollapseomega}}\label{sectioncollapse}

In this section, we sketch the proofs of Theorems \ref{theoremcollapse1} and \ref{theoremcollapseomega}. We begin with Theorem \ref{theoremcollapse1}.

% Warning, this is not right: Note that one may also arrange that in $M$, $\cf(\aleph_1)=\omega$ and $\cf(\aleph_2)=\aleph_1$ by choosing $\lambda$ appropriately. 

% Let us now begin our proof of Theorem \ref{theoremcollapse1}. In the model $N$ constructed in the above proof of Theorem \ref{maintheorem}, we have that $\cf(\kappa)=\omega$, $\cf(\kappa^+)<\kappa$, and $2^\kappa=\theta$. We will now argue that in a symmetric inner model $M$ of a forcing extension of $N$ we have $\cf(\aleph_1)=\cf(\aleph_2)=\omega$ and there is a sequence of distinct subsets of $\aleph_1$ of length $\theta$. Note that one may also arrange that in $M$, $\cf(\aleph_1)=\omega$ and $\cf(\aleph_2)=\aleph_1$ by choosing $\lambda$ appropriately. 

\begin{proof}[Proof of Theorem \ref{theoremcollapse1}]

Suppose
%\marginpar{\tiny Move the specification of $N$ into the proof.}
the model $N$ is such that $\cf(\kappa)^N=\omega$, $\cf(\kappa^+)^N<\kappa$,
and there is, in $N$, a sequence of distinct subsets of
$\kappa$ of length $\theta$. We will now argue that in a symmetric
inner model $M$ of a forcing extension of $N$,
we have $\cf(\aleph_1)=\cf(\aleph_2)=\omega$,
and there is a sequence of distinct subsets of $\aleph_1$
of length $\theta$.

Working in $N$, let $\langle\kappa_n\mid n<\omega\rangle$ be a sequence of inaccessible cardinals less than $\kappa$ which is cofinal in $\kappa$. Let $\P:=\Coll(\omega,{<}\kappa)$, and let $G$ be $N$-generic for $\P$. Let $\P_n:=\Coll(\omega,{<}\kappa_n)$. Standard arguments show that $G_n:=G\cap \P_n$ is $N$-generic for $\P_n$ (see \cite[proof of Theorem 2]{Apter:SuccessorsOfSingularCardinalsAnd}).

%$\P_n:=\Coll(\omega,\kappa_n)$ (see Section \ref{sectionpreliminaries} for definitions). Note that since $N$ and $V$ have the same bounded subsets of $\kappa$, they agree on the definitions of $\P$ and $\P_n$. Clearly we have $\P\cong\prod_{n<\omega}\P_n$ where the product has finite support. We will use the fact that for each $n<\omega$ this product may be factored as $\P\cong \P^*_n\times\P^n$ where $\P^*_n:=\prod_{i\in[0,n]}\P_i$ and $\P^n:=\prod_{i\in[n+1,\omega)}\P_i$ where the second product has finite support. If $G$ is generic for $\P$ then it yields a generic $G^*_n\times G^n$ for $\P^*_n\times\P^n$.

%Let us argue that $\P$ preserves cardinals in $[\kappa,\infty)$ over $N$. (Use the well ordering of each piece $\P_{[0,n]}$ that exists in the Prikry extension $V[G]$ to carry out the standard chain condition argument.)

As in the proof of Theorem \ref{maintheorem}, we let $M$ be the least model of $\ZF$ extending $N$ containing each $G_n$ but not $G$. More formally, let $\mathcal{L}_2$ be the ramified sublanguage of the forcing language associated with $\P$ containing terms $\check{x}$ for each $x\in N$, a unary predicate $\check{N}$ for $N$, and canonical terms $\dot{G}_n$ for each $G_n$. We now define $M$ inductively inside $N[G]$ as follows:

\begin{align*}
M_0&=\emptyset, \\ 
M_\delta&=\bigcup_{\alpha<\delta}M_\alpha\textrm{ for $\delta$ a limit ordinal}, \\ 
M_{\alpha+1}&=\{x\subseteq M_\alpha\mid \textrm{$x$ can be defined over $\langle M_\alpha,\in,c\rangle_{c\in M_\alpha}$} \\
& \hspace{1in}\textrm{using a forcing term $\tau\in\mathcal{L}_2$ of rank $\leq\alpha$}\}, \textrm{ and}\\ 
M&=\bigcup_{\alpha\in\ORD} M_\alpha.
\end{align*}

As before, standard arguments show that $M\models\ZF$. Since $M$ contains $G_n$ for each $n$, it follows that cardinals in $[\omega,\kappa)$ are collapsed to have size $\omega$ and hence $\aleph_1^M\geq\kappa$.
However, standard arguments
(see \cite[Lemma 6.2 and 5.3]{Bull:SuccessiveLargeCardinals})
also show that if $x\in M$ is a set of ordinals,
then $x\in N[G_n]$ for some $n<\omega$.
Since $\Coll(\omega,{<}\kappa_n)$ is canonically well-orderable
in $N$ with order type $\kappa_n$,
the usual proofs show that cardinals and cofinalities
greater than or equal to $\kappa$ are preserved to $N[G_n]$.
Since $\kappa = \aleph^M_1$,
$\cf(\aleph_1)^M=\cf(\aleph_2)^M=\omega$.
It therefore follows that $M\models\lnot\AC_\omega$.
Thus, $M$ is the desired model.

\end{proof}

We remark here that the above proof may be easily adapted to collapse $\kappa$ and $\kappa^+$ to $\delta$ and $\delta^+$ respectively, where $\delta$ is the successor of a regular cardinal, say $\delta=\mu^+$. The main difference between the above proof of Theorem \ref{theoremcollapse1}, and the proof in this more general setting, is that the restricted version of the collapse forcing, call it $\P'_n:=\Coll(\mu,{<}\kappa_n)$, is no longer canonically well-orderable. However, since $N$ and $V$ have the same bounded subsets of $\kappa$, and $V\subseteq N$, it follows that $\P_n'$ can be well-ordered in both $V$ and $N$ with order type less than $\kappa$. In this way, we obtain a model $M$ of $\ZF+\lnot\AC$ in which $\cf(\delta)=\cf(\delta^+)=\omega$ and in which there is a sequence of distinct subsets of $\delta$ of length $\theta$.

% Working in $N$, the relevant collapsing forcing is $\P':=\Coll(\mu,{<}\kappa)$, and we define $\P_n':=\Coll(\mu,{<}\kappa_n)$. One may proceed to carry out the argument, as in the proof of Theorem \ref{theoremcollapse1} above---however, it is no longer the case that $\P_n$ is canonically well-orderable.

Below we present a sketch of our proof of Theorem \ref{theoremcollapseomega}.
As in the above proof sketch of Theorem \ref{theoremcollapse1}, we will argue that in a symmetric inner model $M$ of a forcing extension of $N$, we have $\omega\leq\cf(\aleph_{\omega+1})<\aleph_\omega$, and there is a sequence of distinct subsets of $\aleph_\omega$ of length $\theta$.

\begin{proof}[Proof of Theorem \ref{theoremcollapseomega}]

%We provide only a sketch since the argument is very \marginpar{\tiny Don't need to say it's a sketch so many times.} similar to the proof of Theorem \ref{theoremcollapse1} as well as the proof of Theorem 2 in \cite{Apter:SuccessorsOfSingularCardinalsAnd}.

Let $N$ be constructed so that $\cf(\kappa)^N=\omega$, $\cf(\kappa^+)^N<\kappa$, and there is a sequence of distinct subsets of $\kappa$ of length $\theta$ in $N$. Let $\langle\kappa_i\mid i<\omega\rangle$ be a sequence of inaccessible cardinals cofinal in $\kappa$. Let $\P_0:=\Coll(\omega,{<}\kappa_0)$ and $\P_i:=\Coll(\kappa_{i-1},{<}\kappa_i)$ for $i\in[1,\omega)$. Let $\P:=\prod_{i<\omega}\P_i$, where the product has finite support. For each $n<\omega$, we may factor $\P$ as $\P\cong\P^*_n\times\P^n$, where $\P^*_n:=\prod_{i\in[0,n]}\P_i$ and $\P^n:=\prod_{i\in[n+1,\omega)}\P_i$. Let $G\cong G^*_n\times G^n$ be $N$-generic for $\P$.
As in \cite[proof of Theorem 2]{Apter:SuccessorsOfSingularCardinalsAnd},
each $G_n^*$ is $N$-generic for $\P_n^*$. As before,
we let $M$ be the least model of $\ZF$ extending $N$
containing each $G^*_n$ but not $\la G_n^*\mid n<\omega\ra$.
More formally, let $\mathcal{L}_3$ be the ramified sublanguage
of the forcing language associated with $\P$ containing
terms $\check{x}$ for each $x\in N$,
a unary predicate $\check{N}$ for $N$, and canonical terms $\dot{G}^*_n$
for each $G_n^*$. We now define $M$ inductively inside $N[G]$ as follows:

\begin{align*}
M_0&=\emptyset, \\ 
M_\delta&=\bigcup_{\alpha<\delta}M_\alpha\textrm{ for $\delta$ a limit ordinal}, \\ 
M_{\alpha+1}&=\{x\subseteq M_\alpha\mid \textrm{$x$ can be defined over $\langle M_\alpha,\in,c\rangle_{c\in M_\alpha}$} \\
& \hspace{1in}\textrm{using a forcing term $\tau\in\mathcal{L}_3$ of rank $\leq\alpha$}\},\textrm{ and}\\ 
M&=\bigcup_{\alpha\in\ORD} M_\alpha.
\end{align*}

Since $G^*_n\in M$ for each $n<\omega$, it follows that in $M$, $\aleph_{\omega}\geq\kappa$ and hence $\aleph_{\omega+1}\geq(\kappa^+)^N$. To show that $\kappa=\aleph_\omega$ and $(\kappa^+)^N=\aleph_{\omega+1}$ in $M$, we will use the following lemma.

\begin{lemma}\label{lemmacollapseomega}
If $x$ is a set of ordinals in $M$, then $x\in N[G^*_n]$ for some $n<\omega$.
\end{lemma} 
\noindent For a proof of Lemma \ref{lemmacollapseomega}, one may consult \cite[Lemma 2.1]{Apter:SuccessorsOfSingularCardinalsAnd}.

%=======================================
%=======================================
%=======================================
\begin{comment}

Let us argue that each $\delta_i$ remains a cardinal in $M$. Since by Lemma \ref{lemmacollapseomega} any subset of $\delta_i$ that is in $M$ is in some $N[G^*_n]$, it will suffice to argue that $\delta_i$ remains a cardinal in $N[G^*_n]$. Using a well ordering of $\P^*_n$ in $N$ the same argument works as given at the end of the proof of Theorem \ref{theoremcollapse1}. 

Since each $\delta_i$ remains a cardinal in $M$ it easily follows that $\kappa$ remains a cardinal. Furthermore, Lemma \ref{lemmacollapseomega} implies that cardinals in $[\kappa,\infty)$ are preserved. Thus in $M$ we have $\aleph_\omega=\kappa$ and $\aleph_{\omega+1}=\kappa^+$ and the same cofinal sequences witnessing $(\cf(\kappa)=\cf(\kappa^+)=\omega)^N$ also witness that $(\cf(\aleph_{\omega+1})=\omega)^M$. The sequence of distinct subsets of $\kappa$ of length $\theta$ that is in $N$ is also in $M$. This completes the proof of Theorem \ref{theoremcollapseomega}. 

\end{comment} 
%=======================================
%=======================================
%=======================================

We now argue as in our sketch of the proof
of Theorem \ref{theoremcollapse1}.
Since $N$ and $V$ contain the same bounded subsets of $\kappa$,
and $V\subseteq N$, $\P_n^*$ can be well-ordered in both $V$ and $N$
with order type less than $\kappa$.
Therefore, as before, the usual proofs show that cardinals
and cofinalities greater than or equal to $\kappa$ are preserved.
Furthermore, $M\models\lnot\AC_\omega$ since
$\la G^*_n\mid n<\omega\ra\notin M$.
It follows that $M$ is thus once again the desired model.\end{proof}

We remark that, as in
\cite[Theorem 2]{Apter:SuccessorsOfSingularCardinalsAnd},
in the model $M$ constructed in the above proof of
Theorem \ref{theoremcollapseomega},
$\aleph_\omega$ is a strong limit cardinal.
Also, as we mentioned earlier, by changing the
cardinals to which each $\kappa_i$ is
collapsed, it is possible to collapse $\kappa$ to
$\aleph_{\omega + \omega}$, $\aleph_{\omega^2}$, etc.

\section{An additional result and some open questions}\label{sectionadditional}

In the above results, from $\GCH$ and a supercompact cardinal $\kappa$, we obtain models of $\ZF$ with consecutive singular cardinals, $\kappa$ and $\kappa^+$, in which there is a sequence of distinct subsets of $\kappa$ with any predetermined length---and hence---there is a sequence of distinct subsets of $\kappa^+$ with this same length. This suggests the following question.
\begin{question} \label{question1}
Suppose $\theta_1$ and $\theta_2$ are arbitrary ordinals. Are there models of $\ZF$ with consecutive singular cardinals, $\kappa$ and $\kappa^+$, in which there are sequences of distinct subsets of $\kappa$ and $\kappa^+$ having lengths $\theta_1$ and $\theta_2$ respectively?
\end{question}
\noindent To avoid trivialities, we also require in Question \ref{question1} that there is no sequence of subsets of $\kappa$ of length $\theta_2$.

Let us remark that in Gitik's model in which all uncountable cardinals
are singular (see \cite{Gitik:AllUncountableCardinalsCanBeSingular}),
for every pair of cardinals $\kappa$ and $\kappa^+$,
there is a sequence of distinct subsets
of $\kappa$ of length $\theta_1$ and a sequence
of distinct subsets of $\kappa^+$ of length $\theta_2$,
where $\theta_1$ and $\theta_2$ are ordinals satisfying
$\kappa<\theta_1<\kappa^+<\theta_2<\kappa^{++}$.
% However, as we have already observed, there is no sequence of distinct subsets of $\kappa$ of length $\kappa^+$, nor is there a sequence of distinct subsets of $\kappa^+$ of length $\kappa^{++}$. 
In this sense, Question \ref{question1}
is partially answered by Gitik's model, for some particular $\theta_1$
and $\theta_2$.  However, neither Gitik's model nor our previous theorems address Question \ref{question1}
if we require, e.g., that $\theta_1=\kappa^+$ and $\theta_2\geq\kappa^{++}$.
The following theorem provides more information towards an
answer to Question \ref{question1}, for the case in which
$\kappa<\theta_1<\kappa^+$ and $\theta_2\geq\kappa^+$.

\begin{theorem}\label{theoremadditional}
Suppose $\GCH$ holds, $\kappa < \lambda$
are such that $\kappa$ is $2^\lambda$-supercompact,
and $\lambda$ has cofinality $\omega$
with $\{\alpha<\lambda\mid o(\alpha)\geq\alpha^{+n}\}$
cofinal in $\lambda$ for every $n<\omega$.
Then there is a forcing extension $V[G]$
with a symmetric inner model $N\subseteq V[G]$ of $\ZF$ in which
\begin{enumerate}
\item $\cf(\kappa)=\cf(\kappa^+)=\omega$,
\item there is no $\kappa^+$-sequence of distinct subsets of $\kappa$, and 
\item there is a sequence of distinct subsets of $\kappa^+$ of length $\kappa^{+17}$.
\end{enumerate}
\end{theorem}

Let us remark that the hypotheses of Theorem \ref{theoremadditional}
follow from $\GCH$ and the existence of $\kappa<\delta$
such that $\kappa$ is
$\delta$-supercompact and $\delta$ is $\delta^+$-supercompact.
We also note that by \cite{Gitik:BlowingUpPowerOfASingularCardinal},
in Theorem \ref{theoremadditional}(3) above, one can replace 17
with $\delta+1$ for any $\delta<\aleph_1$.
In addition, note that the hypotheses of Theorem \ref{theoremadditional}
imply that $\lambda$ is a strong limit cardinal,
since it is a limit of inaccessible cardinals.

\begin{proof}[Proof of Theorem \ref{theoremadditional}]
In \cite{Gitik:BlowingUpPowerOfASingularCardinal},
Gitik shows that under these hypotheses on $\lambda$, if $\delta<\aleph_1$,
then there is a forcing notion, call it $\P$,
that preserves cardinals, adds no new bounded subsets to $\lambda$,
and forces $2^\lambda=\lambda^{+\delta+1}$.
It will suffice for us to take $\delta= 16$ so that we achieve (3).

Let $V_0$ satisfy the hypotheses of Theorem \ref{theoremadditional}.
Let $G\subseteq\P$ be $V_0$-generic, and let $V:=V_0[G]$.
Then it follows by Gitik's result that
there is an injection $f:\lambda^{+17}\to P(\lambda)$ in $V$.
Since $\kappa$ is $2^\lambda$-supercompact in $V_0$,
we may let $U\in V_0$ denote a normal fine measure
on $(P_\kappa\lambda)^{V_0}$ satisfying the Menas partition property.
Since $\P$ does not add bounded subsets to $\lambda$,
it follows that $\lambda$ remains a strong limit cardinal
in $V=V_0[G]$, and $\kappa$ remains $\gamma$-supercompact
in $V$ for each cardinal $\gamma<\lambda$.
Indeed, if we let $U\restrict \gamma:= U\cap P(P_\kappa\gamma)$
for each regular cardinal $\gamma<\lambda$,
then $U\restrict\gamma$ is a normal fine measure on $P_\kappa\gamma$
in $V$ satisfying the Menas partition property.

In $V$, let $\langle \gamma_n\mid n<\omega\rangle$
be a sequence of regular cardinals cofinal in $\lambda$,
and let $\Q_{U\restrict\gamma_n}$ denote the supercompact Prikry
forcing over $P_\kappa(\gamma_n)$ defined using $U\restrict\gamma_n$.
Even though $U$ will not be a normal measure on $P_\kappa\lambda$ in $V$,
we can use it in the definition of supercompact Prikry forcing over
$P_\kappa\lambda$. Call this forcing $\Q$.
Let $H$ be $V$-generic for $\Q$,
and let $r_{\gamma_n}$ be the supercompact Prikry sequence
for $\Q_{U\restrict\gamma_n}$ obtained from $H$
as in the proof of Theorem \ref{maintheorem}.
Let $N$ be the smallest inner model of $V[H]$
that contains $r_{\gamma_n}$ for each $n<\omega$
but does not contain $H$. More formally, let
$\mathcal{L}_4$ be the ramified sublanguage
of the forcing language associated
with $\Q$ containing terms $\check{v}$ for each $v\in V$,
a unary predicate $\check{V}$ for $V$, and canonical terms
$\dot r_{\gamma_n}$ for each
$r_{\gamma_n}$. We now define $N$ inductively inside $V[H]$ as follows:

\begin{align*}
N_0&=\emptyset, \\ 
N_\delta&=\bigcup_{\alpha<\delta}N_\alpha\textrm{ for $\delta$ a limit ordinal}, \\ 
N_{\alpha+1}&=\{x\subseteq N_\alpha\mid \textrm{$x$ can be defined over $\langle N_\alpha,\in,c\rangle_{c\in N_\alpha}$} \\
& \hspace{1in}\textrm{using a forcing term $\tau\in\mathcal{L}_4$ of rank $\leq\alpha$}\}, \textrm{ and}\\ 
N&=\bigcup_{\alpha\in\ORD} N_\alpha.
\end{align*}

\begin{lemma}\label{lemmaadditionalkey}
If $x\in N$ is a set of ordinals, then there is an $n<\omega$ such that $x\in V[r_{\gamma_n}]=V_0[G][r_{\gamma_n}]$. 
\end{lemma}

The proof of Lemma \ref{lemmaadditionalkey} is the same as that of Lemma \ref{keylemma} above. Using Lemma \ref{lemmaadditionalkey}, it is straightforward to verify that the conclusions of Theorem \ref{theoremadditional} hold in $N$. Just as in the above proof of Theorem \ref{maintheorem}, it follows from Lemma \ref{lemmaadditionalkey} that $(\kappa^+)^N=\lambda$ and $\cf(\kappa)^N=\cf(\kappa^+)^N=\omega$, which implies that (1) holds in $N$. Furthermore, since the injection $f$ is in $V_0[G]=V\subseteq N$, we conclude that (3) holds in $N$. It remains to show that (2) holds in $N$. Working in $N$, suppose that $\vec{x}=\la x_\alpha\mid\alpha<\kappa^+\ra$ is a sequence of distinct subsets of $\kappa$. Then by Lemma \ref{lemmaadditionalkey}, $\vec{x}\in V[r_{\gamma_n}]$ for some $n<\omega$. This is impossible, since $\lambda=(\kappa^+)^N$
remains a strong limit cardinal in $V[r_{\gamma_n}]$
because $|\Q_{U\restrict\gamma_n}| < \lambda$.
\end{proof}

% The methods of this paper do not seem to be applicable to the following question, whose answer would provide further information toward a resolution of Question \ref{question1}. From large cardinals, is there a model of $\ZF$ with consecutive singular cardinals, $\kappa$ and $\kappa^+$, such that there is a $\kappa^+$-sequence of subsets of $\kappa$ and a $\kappa^{++}$-sequence of subsets of $\kappa^+$? Note that in the paragraph immediately following the statement of Question \ref{question1}, this has been implicitly asked.

The results in this paper suggest the question as to whether
one can prove an Easton theorem-like result,
but for models of $\ZF$ with consecutive singular cardinals. 
Let us state two seemingly very difficult related open questions.

\begin{question}
From large cardinals, is there a model of $\ZF$ in which every cardinal
is singular and in which
for every cardinal $\kappa$,
there is a sequence of $\kappa^{+}$ distinct subsets of $\kappa$?
\end{question}

\begin{question}
From large cardinals, is there a model of $\ZF$ in which every cardinal is singular and in which $\GCH$ fails everywhere in the sense that for every cardinal $\kappa$, there is a sequence of $\kappa^{++}$ distinct subsets of $\kappa$?
\end{question}

Addressing Question \ref{question1}, one would also like to obtain models of $\ZF$ with consecutive singular cardinals, say $\kappa$ and $\kappa^+$, where $\kappa^+$ has uncountable cofinality, $\theta_1<\theta_2$ are cardinals, and $\theta_2\geq\kappa^{+3}$. Notice that Gitik's methods for violating $\GCH$ at ground model singular cardinals do not seem to work for singular cardinals of uncountable cofinality. This suggests the following alternative strategy. Let $\kappa<\lambda$ be the appropriate large cardinals. Using standard techniques, blow up the size of the powerset of $\lambda$ while preserving ``sufficiently many'' of the large cardinal properties of $\kappa$ and $\lambda$. This will allow us to change the cofinality of $\lambda$ to some uncountable cardinal and to change the cofinality of $\kappa$, while simultaneously collapsing all cardinals in the interval $(\kappa,\lambda)$ to $\kappa$. However, the standard forcings for changing to uncountable cofinality at $\lambda$, e.g. Radin or Magidor forcing, will introduce Prikry sequences to unboundedly many cardinals in the interval $(\kappa,\lambda)$ (see \cite{Gitik:Handbook}). By \cite[Theorem 11.1(1)]{CumForMag:SquareScales}, this will introduce nonreflecting stationary subsets of ordinals of cofinality $\omega$ to unboundedly many regular cardinals $\delta$ in the interval $(\kappa,\lambda)$. By \cite[Theorem 4.8]{SolovayReinhardtKanamori} and the succeeding remarks, no cardinal below $\lambda$ is strongly compact up to $\lambda$.  Thus one cannot use the standard forcings for changing the cofinality of $\kappa$ while simultaneously collapsing cardinals in the interval $(\kappa,\lambda)$ to $\kappa$. This suggests that one would like some forcing notion that changes the cofinality of $\lambda>\kappa$ to an uncountable cardinal, and also preserves enough of the original large cardinal properties of $\kappa$ to allow these collapses to occur. As pointed out by the referee of this paper, by the work of Woodin
\cite{Woodin:SuitableExtenderModelsOne} on inner models for supercompact cardinals, it appears as though this is impossible.

% This suggests that one start with two large cardinals, say $\kappa<\lambda$, and then blow up the size of the power set of $\lambda$, while preserving enough of the large cardinal properties of $\kappa$ and $\lambda$ so as to be able to change the cofinality of $\lambda$ to some uncountable cardinal and to change the cofinality of $\kappa$, while simultaneously collapsing all cardinals in the interval $(\kappa,\lambda)$ to $\kappa$.


\begin{thebibliography}{Mat73}

\bibitem[ADK]{ApterDimitiouKoepke:TheFirstMeasurableCardinalCanBe}
Arthur~W. Apter, Ioanna~Matilde Dimitr\'{i}ou, and Peter Koepke.
\newblock The first measurable cardinal can be the first uncountable regular
  cardinal at any successor height.
\newblock {\em \emph{Submitted to the}
%The Mathematical Logic Quarterly
Mathematical Logic Quarterly}.

\bibitem[AH91]{ApterHenle:RelativeConsistencyResultsViaStrongCompactness}
Arthur~W. Apter and James~M. Henle.
\newblock Relative consistency results via strong compactness.
\newblock {\em Fundamenta Mathematicae}, 139(2):133--149, 1991.

\bibitem[Apt83]{Apter:SomeResultsOnConsecutiveLargeCardinals}
Arthur~W. Apter.
\newblock Some results on consecutive large cardinals.
\newblock {\em Annals of Pure and Applied Logic}, 25(1):1--17, 1983.

\bibitem[Apt85]{Apter:SuccessorsOfSingularCardinalsAnd}
Arthur~W. Apter.
\newblock Successors of singular cardinals and measurability.
\newblock {\em Advances in Mathematics}, 55(3):228--241, 1985.

\bibitem[Apt96]{Apter:ADAndPatterns}
Arthur~W. Apter.
\newblock {AD} and patterns of singular cardinals below {$\Theta$}.
\newblock {\em Journal of Symbolic Logic}, 61(1):225--235, 1996.

\bibitem[Bul78]{Bull:SuccessiveLargeCardinals}
Everett~L. Bull.
\newblock Successive large cardinals.
\newblock {\em Annals of Mathematical Logic}, 15(2):161--191, 1978.

\bibitem[CFM01]{CumForMag:SquareScales}
James Cummings, Matthew Foreman, and Menachem Magidor.
\newblock Squares, scales and stationary reflection.
\newblock {\em Journal of Mathematical Logic}, 1(1):35--98, 2001.

\bibitem[Git80]{Gitik:AllUncountableCardinalsCanBeSingular}
Moti Gitik.
\newblock All uncountable cardinals can be singular.
\newblock {\em Israel Journal of Mathematics}, 35(1-2):61--88, 1980.

\bibitem[Git85]{Gitik:RegularCardinalsInModelsOfZF}
Moti Gitik.
\newblock Regular cardinals in models of {ZF}.
\newblock {\em Transactions of the American Mathematical Society},
  290(1):41--68, 1985.

\bibitem[Git02]{Gitik:BlowingUpPowerOfASingularCardinal}
Moti Gitik.
\newblock Blowing up power of a singular cardinal--wider gaps.
\newblock {\em Annals of Pure and Applied Logic}, 116(1-3):1--38, 2002.

\bibitem[Git10]{Gitik:Handbook}
Moti Gitik.
\newblock Prikry-type forcings.
\newblock In Akihiro Kanamori and Matthew Foreman, editors, {\em Handbook of
  Set Theory}, volume~2, chapter~14, pages 1351---1447. Springer, 2010.

\bibitem[Jec03]{Jech:Book}
Thomas Jech.
\newblock {\em Set Theory: The Third Millennium Edition, revised and expanded}.
\newblock Springer, 2003.

\bibitem[Lav78]{Laver:MakingSupercompactnessIndestructible}
Richard Laver.
\newblock Making the supercompactness of $\kappa$ indestructible under
  $\kappa$-directed closed forcing.
\newblock {\em Israel Journal of Mathematics}, 29(4):385---388, 1978.

\bibitem[Mat73]{Mathias:OnSequencesGenericInTheSenseOfPrikry}
Adrian~R.~D. Mathias.
\newblock On generic sequences in the sense of {P}rikry.
\newblock {\em Journal of the Australian Mathematical Society}, 15:409--414,
  1973.

\bibitem[Men76]{Menas:Paper}
Telis~K. Menas.
\newblock A combinatorial property of $P_\kappa\lambda$.
\newblock {\em Journal of Symbolic Logic}, 41(1):225--234, 1976.

\bibitem[SRK78]{SolovayReinhardtKanamori}
Robert~M. Solovay, William~N. Reinhardt, and Akihiro Kanamori.
\newblock Strong axioms of infinity and elementary embeddings.
\newblock {\em Annals of Mathematical Logic}, 13(1):73--116, 1978.

\bibitem[Ste10]{Steel:OutlineOfInnerModelTheory}
John~R. Steel.
\newblock An outline of inner model theory.
\newblock In Akihiro Kanamori and Matthew Foreman, editors, {\em Handbook of
  Set Theory}, volume~3, chapter~19, pages 1595--1684. Springer, 2010.

\bibitem[Woo10]{Woodin:SuitableExtenderModelsOne}
W.~Hugh Woodin.
\newblock Suitable extender models I.
\newblock {\em Journal of Mathematical Logic}, 10(1-2):101--339, 2010.

\end{thebibliography}
\end{document}